\documentclass[10pt]{article}

\usepackage[dvips]{graphicx}
\usepackage[usenames]{color}
\usepackage{amssymb,amsmath,amsfonts}
\usepackage{multirow}
\usepackage{enumerate}
\usepackage{comment}
\usepackage{mdwlist}

\usepackage{epic}
\usepackage{pstricks}
\usepackage{pst-node}
\usepackage{multido}
\usepackage{pstricks-add}

\def\BBox{\kern  -0.2cm\hbox{\vrule width 0.2cm height 0.2cm}}

\newtheorem{theorem}{Theorem}[section]
\newtheorem{lemma}[theorem]{Lemma}

\newtheorem{proposition}[theorem]{Proposition}
\newtheorem{remark}[theorem]{Remark}

\title{Families of Small Regular Graphs of Girth $5$}
\author{M. Abreu$^{1}$, G. Araujo-Pardo$^{2}$, C. Balbuena$^{3}$,
D. Labbate$^{4}$
\thanks{ \footnotesize{\em Email addresses:} marien.abreu@unibas.it (M. Abreu),~
garaujo@matem.unam.mx (G. Araujo),
~ m.camino.balbuena@upc.edu (C. Balbuena), \, \, ~ labbate@poliba.it (D. Labbate)}
 \\[2ex]
$^1${\footnotesize Dipartimento di Matematica, Universit\`{a} degli Studi della
Basilicata,}\\
{\footnotesize Viale dell'Ateneo Lucano, I-85100 Potenza, Italy.} \\$^2$
{\footnotesize Instituto de Matem\'{a}ticas, Universidad Nacional Aut\'{o}noma de M\'{e}xico,} \\
{\footnotesize M\'{e}xico D. F., M\'exico }\\
$^3${\footnotesize Departament de Matem\`atica Aplicada III, Universitat
Polit\`ecnica de Catalunya, }\\
{\footnotesize Campus Nord, Edifici C2, C/ Jordi Girona 1 i 3 E-08034 Barcelona,
Spain.} \\
$^4${\footnotesize Dipartimento di Matematica, Politecnico di Bari,  I-70125 Bari ,
Italy.}
\date{}
}

\begin{document}

\renewcommand\floatpagefraction{.9}
\renewcommand\topfraction{.9}
\renewcommand\bottomfraction{.9}
\renewcommand\textfraction{.1}
\setcounter{totalnumber}{50}
\setcounter{topnumber}{50}
\setcounter{bottomnumber}{50}

\maketitle

\begin{abstract}

In this paper we obtain $(q+3)$--regular
graphs of girth $5$ with fewer vertices than previously known ones for
$q=13,17,19$ and for any prime $q \ge 23$ performing operations of reductions and amalgams on the Levi graph $B_q$
of an elliptic semiplane of type ${\cal C}$. We also obtain a $13$--regular graph of girth $5$ on $236$ vertices from
$B_{11}$ using the same technique.
\end{abstract}

\section{Introduction}\label{intro}

All graphs considered are finite, undirected and simple (without loops or
multiple edges). For definitions and notations not explicitly stated the reader may refer to
\cite{BM}.

Let $G = (V (G),E(G))$ be a graph with vertex set $V = V (G)$ and edge set
$E = E(G)$. The {\em girth} of a graph $G$ is the length $g = g(G)$ of its shortest circuit.
The {\em degree} of a vertex $v \in V$ is the number of vertices adjacent to $v$.
A graph is called {\em $k$--regular} if all its vertices have the same degree $k$, and
{\em bi--regular} or {\em $(k_1,k_2)$--regular} if all its vertices have either degree
$k_1$ or $k_2$.
A {\em $(k,g)$--graph} is a $k$--regular graph of girth $g$ and a
{\em $(k,g)$--cage} is a $(k,g)$--graph with the smallest possible number of vertices.
The necessary condition obtained from the distance partition with respect to a vertex
yields a lower bound $n_0(k, g)$ on the number of vertices of a $(k,g)$--graph,
known as the Moore bound.

$$n_0(k, g) =
\left\{ \begin{array}{ll}
1 + k + k(k - 1) + \ldots + k(k - 1)^{(g-3)/2} & \makebox{if} \, g \, \makebox{is odd}; \\
2(1 + (k - 1) + \ldots + (k - 1)^{g/2 - 1})    & \makebox{if} \, g \, \makebox{is even.}
\end{array}\right.
$$

Biggs \cite{Bi96} calls {\em excess} of a $(k,g)$--graph $G$ the difference $|V (G)| - n_0(k,g)$.
Cages have been intensely studied  since they were introduced by Tutte \cite{T47} in $1947$. Erd\H{o}s
and Sachs \cite{ES63} proved the existence of a $(k,g)$--graph for any value of $k$ and $g$.
Since then, most of the work carried out has been focused on constructing smallest $(k,g)$--graphs
(see e.g. \cite{AFLN06, AFLN08, AGMS07, ABH10,AB11,
BaIt73,B66, BrMcK, E96, FeHi, GH08, M99, OW80, OW81, W82}).
Biggs is the author of an impressive report on distinct methods for constructing cubic cages \cite{B98}.
Royle \cite{Royle} keeps a web-site in which all the cages known so far appear.
More details about constructions on cages can be found in the surveys by Wong \cite{W82},
by Holton and Sheehan \cite[Chapter 6]{HoSh93}, or the recent one by Exoo and Jajcay \cite{EJ08}.

A {\em partial plane} is an incidence structure ${\cal I} = ({\cal P}, {\cal L}, |)$
in which two distinct points are incident with at most one line.
In an incidence structure a {\em flag} is an incident point line pair $p_1 | l_1$,
an {\em anti--flag} is an non--incident point line pair $p_1 \nmid l_1$,
two lines are {\em parallel} if there is no point incident with both,
and, dually, two points are {\em parallel} if there is no line incident with both.

A $v_k$--configuration or a configuration of type $v_k$ is a partial plane
consisting of $v$ points and $v$ lines such that each point and each line are
incident with $k$ lines and $k$ points, respectively. A finite {\em elliptic semiplane}
of order $k - 1$ is a $v_k$--configuration satisfying the following axiom of parallels:
for each anti–-flag $p_1 \nmid l_1$, there exists at most one line $l_2$ incident with
$p_1$ and parallel to $l_1$, and at most one point $p_2$ incident with
$l_1$ and parallel to $p_1$ \cite{Demb,FLN}.

A {\em Baer subset} of a finite projective plane $P$ is either a Baer subplane $B$ or, for a
distinguished point–-line pair $(p, l)$, the union $B(p, l)$ of all lines and points
incident with $p$ and $l$, respectively. We write $B(p|l)$ or $B(p \nmid l)$,
according to the incidence or non--incidence of $p$ and $l$.
It was already known to Dembowski \cite{Demb} that elliptic semiplanes are obtained by deleting a
Baer subset from a projective plane. We call any such elliptic semiplane {\em Desarguesian}
if the projective plane from which it is constructed is so. In \cite{Demb} Dembowski classified elliptic
semiplanes into five types. In this paper we will only be concerned with those of type
$\cal C$, which are ${\cal C}_q = PG(2,q) - B(p|l)$, i.e. the complement of a Baer
subset $B(p|l)$ in a desarguesian projective plane $PG(2,q)$, for each prime power $q$.
Hence the elliptic semiplane of type ${\cal C}_q$ is also a configuration of type $(q^2)_q$.

The {\em Levi graph} or {\em incidence graph} $G$ of an incidence structure
${\cal I} = ({\cal P}, {\cal L}, |)$, is a bipartite graph with $V(G) = V_1 \cup V_2$,
where $V_1={\cal P}$ and $V_2={\cal L}$ and two vertices are adjacent in $G$ if and
only if the corresponding point and line are incident in $\cal I$.
Recall that the Levi graph of a finite projective plane is a $(k,6)$--cage, attaining
Moore's bound, i.e. these are Moore graphs \cite{W82}.

In this paper we obtain $(q+3)$--regular
graphs of girth $5$ with fewer vertices than previously known ones (cf. \cite{J05,F10}) for
$q=13,17,19$ and for any prime $q \ge 23$ performing operations of reductions (cf. Section \ref{reduc}) on the Levi graph $B_q$
of ${\cal C}_q$ and then amalgams with bi--regular graphs (cf. Section \ref{amal}) into the obtained reduced graph or $B_q$ itself.
We also obtain a new $13$--regular graph of girth $5$ on $236$ vertices from $B_{11}$ using the same technique.



\section{Preliminaries}\label{prel}

Throughout the paper we will use the following notation when dealing with the elliptic semiplane of type ${\cal C}_q$.

%

In $PG(2,q)$, choose $p$ and $l$ to be the point and line at infinity, respectively.
Then, in ${\cal C}_q$ it is possible to choose the affine coordinates
$(x,y)$, for the points, and $[m,b]$ for the lines $\{x,y,m,b\}\in GF(q)$,
which imply that the incidence between a point and a line is given by the equation
$y=mx+b$. Recall that in ${\cal C}_q$ vertical lines have been deleted from $PG(2,q)$
along with the point at infinity, the line at infinity and all its points.


Define the sets $P_i=\{(i,y) |\ y\in GF(q)\}$ for $i\in GF(q)$ and
$L_j=\{L_j=\{[j,b] |\ b\in GF(q)\}$ for $j\in GF(q)$. These sets
correspond to the partition of the points and lines of ${\cal C}_q$ into parallel
classes, according to the axiom of parallels for elliptic semiplanes.
Note also that if $(x,y) | [m,b]$ then $(x,y+a) | [m,b+a]$ for any $a \in GF(q)$.

The following properties of the Levi graph $B_q$ of ${\cal C}_q$ are well known  and they will be fundamental throughout the paper.

\begin{proposition}\label{BqProp}

Let $B_q$ be the Levi graph of ${\cal C}_q$ then:

\begin{enumerate}[(i)]
\item It is $q$--regular, bipartite, vertex transitive, of order $2q^2$ and has girth $6$;
\item It admits a partition $V_1 = \displaystyle \bigcup^{q-1}_{i=0} P_i$
and $V_2 = \displaystyle \bigcup^{q-1}_{j=0} L_j$ of its vertex set;
\item Each block $P_i$ is connected to each block $L_j$ by a perfect matching,
for $i,j \in GF(q)$;
\item Each vertex in $P_0$ and $L_0$ is connected {\em straight} to all its
neighbours in $B_q$, meaning that for $p=(0,y)$, $N(p)=\{[i,y] | i \in GF(q) \}$
and analogously for $l=[0,b]$, $N(l)=\{(j,b) | j \in GF(q) \}$;
\item The other matchings between $P_i$ and $L_i$ are {\em twisted} and the rule can be
defined algebraically in $GF(q)$.
\end{enumerate}
\end{proposition}

%

For further information regarding these properties and for constructions of the adjacency matrix
of $B_q$ as a block $(0,1)$--matrix please refer to \cite{ALB:10, B08}.

\section{Reductions}\label{reduc}
In this section we will describe two reduction operations that we perform on the graph $B_q$.


\

\noindent{\sc Reduction 1} Remove vertices from $P_0$ and $L_0$.

Let $T \subseteq S \subseteq GF(q)$, $S_0 = \{(0,y) | y \in S\}  \subseteq P_0$,
$T_0 = \{[0,b] | b \in T\}  \subseteq L_0$ and $B_q(S,T)=B_q-S_0-T_0$.


\begin{lemma}\label{Red1}
Let $T \subseteq S \subseteq GF(q)$. Then $B_q(S,T)$ is bi--regular with degrees $(q-1,q)$ of order $2q^2-|S|-|T|$.
Moreover, the vertices $(i,t) \in V_1$ and $[j,s] \in V_2$, for each $i,j \in GF(q) -\{0\}$,
$s \in S$ and $t \in T$ are the only vertices of degree $q-1$ in $B_q(S,T)$, together with
$[0,s] \in V_2$ for $s \in S-T$ if $T \subsetneq S$.
\end{lemma}

\begin{proof}
It is an immediate consequence of
Remark \ref{BqProp} $(i)$, $(v)$.
\end{proof}

\noindent{\sc Reduction 2} Remove pairs of blocks $(P_i,L_i)$ from $B_q$.

Let $u \in \{1, \ldots, q-1 \}$. Define $B_q(u)=B_q-\bigcup^u_{i=1} (P_{q-i} \cup L_{q-i})$
the graph obtained from $B_q$ by deleting the last $u$ pairs of blocks of vertices $P_i, L_i$.
and $B_q(S,T,u)=B_q-S_0-T_0-\bigcup^u_{i=1} (P_{q-i} \cup L_{q-i})$.

\begin{lemma}\label{Red2}
Let $u \in \{0, \ldots, q-1 \}$. Then, the graph $B_q(u)$ is $(q-u)$--regular of order $2(q^2-qu)$
and the graph $B_q(S,T,u)$ is bi--regular with degrees $(q-u-1,q-u)$ and order $2(q^2-qu)-|S|-|T|$.
Moreover, the vertices $(i,t) \in V_1$ and $[j,s] \in V_2$, for each $i,j \in GF(q)$,
$s \in S$ and $t \in T$ are the only vertices of degree $q-u-1$ in $B_q(S,T,u)$,
together with $[0,s] \in V_2$ for $s \in S-T$ if $T \subsetneq S$.
\end{lemma}

\begin{proof}
It is immediate from
Remark \ref{BqProp} $(i)$, $(iv)$ and Lemma \ref{Red1}. \end{proof}

Note that,  $B_q(u)=B_q$ and $B_q(S,T,u)=B_q(S,T)$ when $u=0$.

\section{Amalgams}\label{amal}
In this section we will describe amalgam operations that can be performed on the reduced graph $B_q(S,T,u)$ or on $B_q$ itself.

Let $\Gamma_1$ and $\Gamma_2$ be two graphs of the same order and with the same label on their vertices.
In general, an {\em amalgam of $\Gamma_1$ into $\Gamma_2$} is a graph obtained adding all the edges of $\Gamma_1$ to $\Gamma_2$.

\

Let $P_i$ and $L_i$ be defined as in Section \ref{prel}.
Consider the graph $B_q(S,T,u)$,  for some $T \subseteq S \subseteq GF(q)$ and $u \in \{0, \ldots, q-1 \}$.
Let $S_0 \subseteq P_0$, $T_0 \subseteq L_0$ as in Reduction $1$, and let $P_0':=P_0-S_0$ and $L_0':=L_0-T_0$ be
the blocks in $B_q(S,T,u)$ of order $q-|S|$ and $q-|T|$, respectively.


Let $H_1$, $H_2$, $G_i$, for $i=1,2$, be graphs of girth at least $5$ and order $q-|S|$, $q-|T|$ and $q$, respectively.
Let $H_1$ be a $k$--regular graph. If $|S|=|T|$, let $H_2$ be $k$--regular and otherwise let it be $(k,k+1)$--regular,
with $|S-T|$ vertices of degree $k+1$.  If $T = \emptyset$, let $G_1$ be a $k$--regular graph and otherwise let it be
$(k,k+1)$--regular with $|T|$ vertices of degree $k+1$.
Finally, let $G_2$ be a $(k,k+1)$--regular with $|S|$ vertices of degree $k+1$.

We define $B^*_q(S,T,u)$ to be the {\em amalgam} of $H_1$ into $P_0'$,
$H_2$ into $L_0'$, $G_1$ into $P_i$ and $G_2$ into $L_i$, for $i \in \{1, \ldots, q-u-1\}$
and $u \in \{0, \ldots, q-2\}$. We also define $B^*_q(S,T,q-1)$ to be the amalgam of $H_1$ into $P_0'$,
$H_2$ into $L_0'$.

\


%
To simplify notation in our results, we label $P_i$ and $L_i$ as in Section \ref{prel},
but assume that the labellings of $H_1,H_2,G_1$ and $G_2$, correspond
to the second coordinates of $P_0', L_0', P_i$ and $L_i$ respectively for $i \in \{1, \ldots, q-u-1\}$
and $u \in \{0, \ldots, q-2\}$.
Suppose also that the vertices of degree $k+1$, if any, in $H_2,G_1$ and $G_2$ are labelled
in correspondence with the second coordinates of $S-T$, $T$ and $S$, respectively.

%

With such a labelling, let $ab$ be an edge in $H_1,H_2,G_1$ or $G_2$, and define the {\em weight}
or the {\em Cayley Color} of $ab$ to be $\pm (b-a) \in {\mathbb Z}^*_q$. Let ${\cal P}_{\omega}$ be the
set of weights in $H_1$ and $G_1$, and let ${\cal L}_{\omega}$ be the set of weights in $H_2$ and $G_2$.

\

The following result is a special case of \cite[Theorem 2.8]{F10} for the coordinates we have chosen
for ${\cal C}_q$ (cf. Section \ref{prel}). On the other hand, it generalizes such a Theorem since we delete vertices from
$P_0$ and $L_0$, pairs of blocks $P_i, L_i$ and amalgam with graphs which are not regular,
but chosen in such a way that the obtained amalgam is regular.

\begin{theorem}\label{AmalgamThm} Let $T \subseteq S \subseteq GF(q)$, $u \in \{0, \ldots, q-1 \}$.
Let $H_1,H_2,G_1$ and $G_2$ be defined as above and suppose that the weights
${\cal P}_{\omega} \cap {\cal L}_{\omega} = \emptyset$.
Then the amalgam $B^*_q(S,T,u)$ is a $(q+k-u)$--regular graph of girth at least $5$ and order $2(q-u)-|S|-|T|$.
\end{theorem}

\begin{proof}
The order and the regularity of $B^*_q(S,T,u)$ follow from Lemma \ref{Red2} and the choice of $H_1,H_2,G_1$ and $G_2$.
Note that the vertices of $L_i$, with degree $q-u-1$ in $B_q(S,T,u)$, have degree $k+1$ in $G_2$,
which add up to to degree $q+k-u$ in $B^*_q(S,T,u)$, for $i \in \{1, \ldots, q-u-1\}$.
Similarly for the vertices in $L_0$ and for those in $P_i$, for $i \in \{1, \ldots, q-u-1\}$.


Let $C$ be the shortest circuit in $B^*_q(S,T,u)$ and suppose, by contradiction, that $|C| \leq 4$.
Therefore, $C=(xyz)$ or $C=(wxyz)$.
Since $B_q$ has girth $6$ and $H_1,H_2,G_1,G_2$ have girth at least $5$, then $C$ cannot be completely contained in $B_q$ or
in some $H_i$ or $G_i$ for $i=1,2$.
Then, w.l.o.g. the path $xyz$ in $C$ is such that $x,y \in P_i$ and $z \in L_m$ for some $i,m \in GF(q)$.
Since the edges between $P_i$ and $L_m$ form a matching, then $xz \notin E(B_q)$ and hence $xz \notin E(B_q^*(S,T,u))$.
Thus $|C| > 3$ and we can assume $|C|=4$ and $C=(wxyz)$.

If $w \in P_i$, by the same argument, $wz \notin E(B_q^*(S,T,u))$ and we have a contradiction.
There are no edges between $P_i$ and $P_j$ in $B^*_q(S,T,u)$, so $w \notin P_j$ for $j \in GF(q) - \{i\}$,
which implies that $w \in L_n$ for some $n \in GF(q)$. If $n \ne m$, we have a contradiction since there
are no edges between $L_m$ and $L_n$ in $B^*_q(S,T,u)$. Therefore $x,y \in P_i$ and $w,z \in L_m$.
Let $x=(i,a), y=(i,b), z=[m,c]$ and $w=[m,d]$ as in the labelling chosen in Section \ref{prel}.
Then $wx,yz \in E(B^*_q(S,T,u))$ imply that $a=m \cdot i + d$ and $b=m \cdot i + c$, respectively,
which give $b - a = c - d$. On the other hand $xy, wz \in E(B^*_q(S,T,u))$ implies that $ab \in E(H_1) \cup E(G_1)$ and $cd \in E(H_2) \cup E(G_2)$,
so $\pm (a-b) \in {\cal P}_{\omega}$ and $\pm (c-d) \in {\cal L}_{\omega}$, a contradiction, since by hypothesis
${\cal P}_{\omega} \cap {\cal L}_{\omega} = \emptyset$.
\end{proof}

\begin{remark}\label{sharp5} In most cases the graph $B^*_q(S,T,u)$ has girth exactly $5$. We describe two cases that we will use in Sections \ref{newgirth5} and \ref{smallcases}.

(i) If some $H_i$ or $G_i$ contains a $5$--circuit, for $i \in \{1,2\}$, then so does $B^*_q(S,T,u)$.

(ii) Let $t \in GF(q)$ be the smallest weight of an edge in some $H_i$, say w.l.o.g. in $H_1$.
If $t=1$ and $u < q-1$ then $((0,i),(0,j),[1,j],(1,i),[0,i])$ is a $5$--circuit in $B^*_q(S,T,u)$.
If $t > 1$ then $((0,i),(0,j),[1,j],(t,i),[0,i])$ is a $5$--circuit in $B^*_q(S,T,u)$ as long as the
$t^{th}$--pair of blocks from $B^*_q(S,T)$ is not deleted, i.e. $u < q - t$.

\end{remark}

All the graphs constructed in the next sections have girth exactly $5$ since either
some $H_i$ or $G_i$ contains a $5$--circuit, for $i \in \{1,2\}$, or  $1 \in P_\omega$.



\section{New Regular Graphs of Girth $5$}\label{newgirth5}

In this section we will construct new $(q+3)$--regular graphs of girth $5$,
for any prime $q \ge 23$, applying reductions and amalgams to the graph $B_q$. In each case we will specify the sets $S$ and $T$
of vertices to be deleted from $P_0$ and $L_0$ and the graphs $H_1,H_2,G_1,G_2$ to be used for the amalgam into
$B^*_q(S,T,u)$. For $u=0$, all the graphs $B^*_q(S,T,u)$ constructed in this section have two vertices less than the ones
that appear in $\cite{J05,F10}$.

Recall that every prime $q$ is either congruent to $1$ or $5$ modulo $6$. We will now treat these two cases separately,
when $q=6n+1$ or $q=6n+5$ is a prime.

\subsection {Construction for primes $q=6n+1$}

Throughout this subsection we will consider  $n \ge 5$. The smaller cases will be treated in Section \ref{smallcases}.
Let $H_1$ and $H_2$ be two graphs of order $q-1$ with the vertices labeled from $1$ through $6n$,
and partitioned into $W_1=\{1,2,..,3n\}$ and $W_2=\{3n+1,..,6n\}$.

Define the set of edges $E(H_1) = A_1 \cup B_1 \cup C_1$ as follows:

\

\resizebox{11cm}{!}{
\begin{minipage}{13cm}
\begin{tabular}{cll}
Set & \quad \quad \quad \quad \quad \quad Edges & \quad \quad Description \\
\hline $A_1$ & $\{(i,i+1) | i = 1, \ldots ,3n-1 \} \cup \{ (3n,1)\}$ & $(3n)$--circuit with weights \\
                                                                 & & $1$ and $3n-1$ \\
\hline $B_1$ & $\{(i,i+2) | i = 3n+1, \ldots ,6n-2 \} $ & one or two circuits \\
             & $\cup \{ (6n-1,3n+1), (6n,3n+2)\}$     & according to the parity of $n$, \\
             &                                          & with weights $2$ and $3n-2$\\
\hline $C_1$ & $\{(i,3n+i) | i = 1,\ldots ,3n \}$       & Prismatic edges between $W_1$ and $W_2$ \\
                                                      & & of weight $3n$

\end{tabular}
\end{minipage}
}

\

The graph $H_1$ is cubic and has weights $\pm \{1,2,3n-2,3n-1,3n\}$.

\begin{lemma}\label{H1g5}
The graph $H_1$ has girth $5$.
\end{lemma}

\begin{proof}
Let $C$ be the shortest circuit in $H_1$. If $C$ is a subgraph of either $H_1[W_1]$ or $H_1[W_2]$ then $|C| \ge 5$,
since $H_1[W_1]$ has girth at least $15$ and $H_1[W_2]$ has girth at least $9$. Otherwise, there is a path
$xyz$ in $C$ is such that either $x,y \in W_1$ and $z \in W_2$ or $x \in W_1$ and $y,z \in W_2$.
The first case has the following subcases:
\begin{enumerate}[(i)] \setlength{\itemsep}{-2mm}
\item $x=1$, $y=3n$, $z=6n$
\item $x=i$, $y=i-1$, $z=3n+i-1$, for $i=2, \ldots 3n$
\item $x=i$, $y=i+1$, $z=3n+i+1$, for $i=1, \ldots 3n-1$
\item $x=3n$, $y=1$, $z=3n+1$
\end{enumerate}
The second case has similar subcases. If we show that
$z \notin N_{H_1}(x)$ then $|C| \ne 3$,
and if $y = N_{H_1}(x) \cap N_{H_1}(z)$ then $|C| \ne 4$.
In subcase (i) the neighbourhoods of $x$ and $z$ in $H_1$ are $N_{H_1}(x)=\{2,3n,3n+1\}$
and $N_{H_1}(z)=\{3n,3n+2,6n-2\}$, respectively.
Thus, $z \notin N_{H_1}(x)$ and $y = N_{H_1}(x) \cap N_{H_1}(z)$.
Hence, $|C| \ge 5$. All the other cases are analogous.
The circuit $(1,2,3,3n+3,3n+1)$ is a $5$--circuit in $H_1$.

\end{proof}

Define the set of edges $E(H_2) = A_2 \cup B_2 \cup C_2$ as follows:

\

\resizebox{11cm}{!}{
\begin{minipage}{13cm}
\begin{tabular}{cll}
Set & \quad \quad \quad \quad \quad \quad Edges & \quad \quad Description \\
\hline $A_2$ & $\{(i,i+3) | i = 1,\ldots,3n-3 \}$ & Three $n$--circuit with weights \\
             & $\cup \{ (3n-2,1), (3n-1,2), (3n,3)\}$  & $3$ and $3n-3$ \\
\hline $B_2$ & $\{(i,i+4) | i = 3n+1,\ldots,6n-4 \} $                       & One, two or four circuits \\
             & $\cup \{ (6n-3,3n+1), (6n-2,3n+2), $                         & according to the congruency of $3n$ modulo $4$, \\
             & $(6n-1,3n+3), (6n,3n+4)\}$                                   & with weights $4$ and $3n-4$\\
\hline $C_2$ & $\{(i,3n+4+i) | i = 1,\ldots,3n-4 \}$                          & Prismatic edges between $W_1$ and $W_2$ \\
             & $\cup \{(3n-3, 3n+1), (3n-2, 3n+2), $                        & of weights $4$ and $3n+4 \equiv 3n-3 \, \mod \, q$ \\
             & $(3n-1, 3n+3), (3n, 3n+4) \}$                                &
\end{tabular}
\end{minipage}
}

\

The graph $H_2$ is cubic and has weights $\pm \{3,4,3n-4,3n-3\}$.

\begin{lemma}\label{H2g5}
The graph $H_2$ has girth at least $5$.
\end{lemma}

\begin{proof}
Similar to the proof of Lemma \ref{H1g5}.
\end{proof}

\begin{lemma}\label{TchInd}
Let $G$ be a graph of girth at least $5$. Let $x_1x_2, x_3x_4 \in E(G)$
be two independent edges of $G$ such that $N(x_i) \cap N(x_j) = \emptyset$,
for all $i,j \in \{1,2,3,4\}$, $i \ne j$.
Let $G' = G- \{x_1x_2,x_3x_4\} \cup \{(v,x_i) | i=1,2,3,4\}$ be the graph of order $|V(G)|+1$, where
$v=V(G')-V(G)$.
Then $G'$ has girth at least $5$.
\end{lemma}

\begin{proof}
Let $C$ be the shortest circuit in $G'$. If $E(C) \subset E(G)$ then, by hypothesis, $|C| > 4$.
Otherwise $v \in V(C)$ and $x_ivx_j$ is a path in $C$ for some $i,j \in \{1,2,3,4\}$, $i \ne j$.
In $G'$ the set $\{x_i | i = 1,2,3,4\}$ is independent, so $|C| > 3$.
By hypothesis, $N(x_i) \cap N(x_j) = v$ in $G'$ and hence $|C| > 4$.
\end{proof}

Let $G_1$ be a graph on $q$ vertices labelled from $0$ through $q-1$ and
defined as follows $G_1 := H_1 - \{(1,3n), (\lfloor \frac{3n+1}{2} \rfloor, 3n + \lfloor \frac{3n+1}{2} \rfloor)\}
\cup \{ (0,1), (0,\lfloor \frac{3n+1}{2} \rfloor), (0,3n),$ $(0, 3n + \lfloor \frac{3n+1}{2} \rfloor) \}$

\begin{lemma}\label{G_1g5}
The graph $G_1$ has girth at least $5$.
\end{lemma}

\begin{proof}
The edges $e_1=(1,3n)$ and $e_2=(\lfloor \frac{3n+1}{2} \rfloor, 3n + \lfloor \frac{3n+1}{2} \rfloor)$ are
independent in $H_1$. The neighbourhoods of the endvertices of $e_1$ and $e_2$ are:
$$
\begin{array}{ll}
N(1) & = \{2,3n,3n+1\}; \\
N(\lfloor \frac{3n+1}{2} \rfloor) &
= \{\lfloor \frac{3n+1}{2} \rfloor -1 ,\lfloor \frac{3n+1}{2} \rfloor + 1,3n+\lfloor \frac{3n+1}{2} \rfloor\}; \\
N(3n) & =\{1,3n-1,6n\}; \\
N(3n + \lfloor \frac{3n+1}{2} \rfloor) &
= \{ 3n + \lfloor \frac{3n+1}{2} \rfloor -1 ,
3n + \lfloor \frac{3n+1}{2} \rfloor + 1,
\lfloor \frac{3n+1}{2} \rfloor\};
\end{array}
$$

\noindent which satisfy the hypothesis of Lemma \ref{TchInd}.
Since $G_1$ is constructed from $H_1$ as $G'$ from $G$ in Lemma \ref{TchInd},
we can conclude that $G_1$ has girth at least $5$.
\end{proof}

All together the weights of $H_1$ and $G_1$  modulo $p$ give

\begin{equation} \label{EqPw1}
{\cal P}_\omega := \left \{
\begin{array}{ll}
\pm \{ 1, 2,  \frac{3n+1}{2},  3n-2,  3n-1,  3n\} & \text{if } n  \text{ is odd} \\
\pm \{ 1, 2,  \frac{3n}{2},  \frac{3n+2}{2},  3n-2,  3n-1,  3n\} & \text{if } n  \text{ is even}
\end{array}
\right.
\end{equation}


Let $G_2$ be a graph on $q$ vertices labelled from $0$ through $q-1$ and
defined as follows:

\[
G_2 := \left \{
\begin{tabular}{ll}
$H_2 - \{(3,22), (5,24)\} \cup \{ (0,3), (0,5), (0,22),(0,24) \}$          & if $n  = 5$ \\
$H_2 - \{(3,3n+7), (4,3n+8)\}$          & \multirow{2}{*}{if $n  \ge 6$} \\
$\cup \{ (0,3), (0,4), (0,3n+7),(0, 3n+8) \}$ &   \\
\end{tabular}
\right.
\]

\

Note that for $n=5$ the edge $(0,3n+8)=(0,23)$ has weight $-8$ which lies already in ${\cal P}_\omega$
and Theorem \ref{AmalgamThm} cannot be applied. This is why, in the definition of $G_2$, we choose
to delete the edge $(5,24)$ from $H_2$, instead of $(4,3n+8)=(4,23)$.

\

\begin{lemma}\label{G_2g5}
The graph $G_2$ has girth at least $5$.
\end{lemma}

\begin{proof}
First suppose $n \ge 6$. As in Lemma \ref{G_1g5}, the edges $(3,3n+7), (4,3n+8)$
are independent in $H_2$ and the neighbourhoods
$$
\begin{array}{ll}
N(3) & = \{6,3n,3n+7\}; \\
N(\lfloor \frac{3n+1}{2} \rfloor) &
= \{\lfloor \frac{3n+1}{2} \rfloor -1 ,\lfloor \frac{3n+1}{2} \rfloor + 1,3n+\lfloor \frac{3n+1}{2} \rfloor\}; \\
N(3n) & =\{1,3n-1,6n\}; \\
N(3n + \lfloor \frac{3n+1}{2} \rfloor) &
= \{ 3n + \lfloor \frac{3n+1}{2} \rfloor -1 ,
3n + \lfloor \frac{3n+1}{2} \rfloor + 1,
\lfloor \frac{3n+1}{2} \rfloor\};
\end{array}
$$
\noindent which satisfy the hypothesis of Lemma \ref{TchInd}.
Since $G_2$ is constructed from $H_2$ as $G'$ from $G$ in Lemma \ref{TchInd},
$G_2$ has girth at least $5$.

\noindent Similarly for $n=5$.
\end{proof}

All together the weights of $H_2$ and $G_2$  modulo $q$ give

\begin{equation} \label{EqLw1}
{\cal L}_\omega := \left \{
\begin{array}{ll}
\pm \{ 3, 4, 7, 9, 11, 12\} & \text{if } n  = 5 \\
\pm \{ 3, 4, 3n-7, 3n-6, 3n-4,  3n-3\} & \text{if } n  \ge 6
\end{array}
\right.
\end{equation}

\

\begin{theorem}\label{p1mod6} Let $q$ be a prime such that $q=6n+1$, $n \ge 2$. 
Then, there is a $(q+3-u)$--regular graph of girth $5$ and order $2(q^2-u-1)$, for each $0 \le u \le q-1$.
\end{theorem}

\begin{proof}
We treat the cases $n=2, 3$ in Section \ref{smallcases}.
For $n=4$, $q=6n+1=25$ is not a prime, therefore we can assume that $n \ge 5$.

Let $S=T=\{0\}$ and choose $H_i$, $G_i$ for $i=1,2$ as previously described in this subsection.
Lemmas \ref{H1g5}, \ref{H2g5},  \ref{G_1g5}, \ref{G_2g5} together with (\ref{EqPw1})
and (\ref{EqLw1}) imply that the hypothesis of Theorem \ref{AmalgamThm}
are satisfied. Therefore, the graphs $B^*_q(S,T,u)$ are $(q+3-u)$--regular of girth $5$
and order $2(q^2-u-1)$ for each $0 \le u \le q-1$. Note that the girth of
$B^*_q(S,T,u)$ is exactly $5$ because $H_1$ has girth $5$ (cf. Remark \ref{sharp5}).
\end{proof}


\subsection {Construction for primes $q=6n+5$}\label{constr6n+5}

We consider  $n \ge 3$ throughout this subsection and we treat smaller cases in Section \ref{smallcases}.
Let $H_1$ and $H_2$ be two graphs of order $q-1$ with the vertices labelled from $1$ through $6n+4$,
and partitioned into $W_1=\{1,2,..,3n+2\}$ and $W_2=\{3n+3,..,6n+4\}$.

Define the set of edges $E(H_1) = A_1 \cup B_1 \cup C_1$ as follows:

\

\resizebox{10.5cm}{!}{
\begin{minipage}{13cm}
\begin{tabular}{cll}
Set & \quad \quad \quad \quad \quad \quad Edges & \quad \quad Description \\
\hline $A_1$ & $\{(i,i+1) | i = 1, \ldots ,3n+1 \} \cup \{ (3n+2,1)\}$ & $(3n+2)$--circuit with weights \\
                                                                 & & $1$ and $3n+1$ \\
\hline $B_1$ & $\{(i,i+2) | i = 3n+3, \ldots ,6n + 2 \} $ & one or two circuits \\
             & $\cup \{ (6n+3,3n+3), (6n+4,3n+4)\}$     & according to the parity of $n$, \\
             &                                          & with weights $2$ and $3n$\\
\hline $C_1$ & $\{(i,3n+i+2) | i = 1,\ldots ,3n+2 \}$   & Prismatic edges between $W_1$ and $W_2$ \\
                                                      & & of weight $3n+2$

\end{tabular}
\end{minipage}
}

\

The graph $H_1$ is cubic and has weights $\pm \{1,2,3n,3n+1,3n+2\}$.

\

%

Define the set of edges $E(H_2) = A_2 \cup B_2 \cup C_2$ as follows:

\

\resizebox{11cm}{!}{
\begin{minipage}{13cm}
\begin{tabular}{cll}
Set & \quad \quad \quad \quad \quad \quad Edges & \quad \quad Description \\
\hline $A_2$ & $\{(i,i+3) | i = 1,\ldots,3n-1 \}$ & One $3n+2$--circuit with weights \\
             & $\cup \{ (3n,1), (3n+1,2), (3n+2,3)\}$  & $3$ and $3n-1$ \\
\hline $B_2$ & $\{(i,i+4) | i = 3n+3,\ldots,6n \} $                       & One, two or four circuits \\
             & $\cup \{ (6n+1,3n+3), (6n+2,3n+4), $                       & according to the congruency of $n$ modulo $4$, \\
             & $(6n+3,3n+5), (6n+4,3n+6)\}$                               & with weights $4$ and $3n-2$\\
\hline $C_2$ & $\{(i,3n+i+6) | i = 1,\ldots,3n-2 \}$                      & Prismatic edges between $W_1$ and $W_2$ \\
             & $\cup \{(3n-1, 3n+3), (3n, 3n+4), $                        & of weights $4$ and $3n+6 \equiv 3n-1 \, \mod \, q$ \\
             & $(3n+1, 3n+5), (3n+2, 3n+6) \}$                            &
\end{tabular}
\end{minipage}
}

\

The graph $H_2$ is cubic and has weights $\pm \{3,4,3n-2,3n-1\}$.

\

%

Let $G_1$ be a graph on $q$ vertices labeled from $0$ through $q-1$ and defined as follows:

\[
G_1 := \left \{
\begin{tabular}{ll}
$H_1 - \{(1,12), (6,17)\} \cup \{ (0,1), (0,6), (0,12),(0,17) \}$          & if $n  = 3$ \\
$H_1 - \{(1,3n+3), (\lfloor \frac{3n+1}{2} \rfloor, 3n + 2 + \lfloor \frac{3n+1}{2} \rfloor)\}$          & \multirow{2}{*}{if $n  \ge 4$} \\
$\cup \{ (0,1), (0,\lfloor \frac{3n+1}{2} \rfloor),(0,3n+3),(0, 3n + 2 + \lfloor \frac{3n+1}{2} \rfloor) \}$ &   \\
\end{tabular}
\right.
\]

\

%

Note that for $n=3$ the independent edges $(1,3n+3)=(1,12)$ and $(\lfloor \frac{3n+1}{2} \rfloor, 3n + 2 + \lfloor \frac{3n+1}{2} \rfloor)=(5,6)$
of $H_1$ have a common neighbour, namely $N_{H_1}(12) \cap N_{H_1}(16)=\{14\}$, and Lemma \ref{TchInd} cannot be applied.
This is why we choose the independent edges $(1,12)$ and $(6,17)$ in $H_1$ with pairwise disjoint neighbourhoods to define $G_1$.

\

All together the weights of $H_1$ and $G_1$  modulo $q$ give

\begin{equation} \label{EqPw2}
{\cal P}_\omega := \left \{
\begin{array}{ll}
\pm \{ 1, 2, 6, 9, 10, 11\}                                      & \text{if } n=3 \\
\pm \{ 1, 2,  \frac{3n+1}{2}, \frac{3n+5}{2}, 3n,  3n+1,  3n+2\} & \text{if } n  \text{ is odd and } n \ge 5\\
\pm \{ 1, 2,  \frac{3n}{2},   \frac{3n+6}{2}, 3n,  3n+1,  3n+2\} & \text{if } n  \text{ is even}
\end{array}
\right.
\end{equation}

Let $G_2$ be a graph on $q$ vertices labeled from $0$ through $q-1$ and
defined as follows $G_2 := H_2 - \{(3,3n+9), (4,3n+10)\}
\cup \{ (0,3), (0,4), (0,3n+9),$ $(0, 3n+10) \}$.

%

All together the weights of $H_2$ and $G_2$  modulo $q$ give
\begin{equation} \label{EqLw2}
{\cal L}_\omega := \pm \{ 3, 4,  3n-5,  3n-4,  3n-2,  3n-1\}.
\end{equation}

\begin{lemma}\label{H12G12g5b}
The graphs $H_1,H_2,G_1$ and $G_2$ have girth at least $5$.
\end{lemma}

\begin{proof}
Similar to Lemmas \ref{H1g5}, \ref{H2g5}, \ref{G_1g5}, \ref{G_2g5}.
\end{proof}

Note that in general, the girth of $H_1$ is exactly $5$, since $(1,2,3,3n+5,3n+3)$ is a $5$--circuit in $H_1$.


\begin{theorem}\label{p5mod6} Let $q$ be a prime such that $q=6n+5$, for $n \ge 3$.
Then, there is a $(q+3-u)$--regular graph of girth $5$ and order $2(q^2-u-1)$ for each $0 \le u \le q-1$.
\end{theorem}

\begin{proof}
%
Let $S=T=\{0\}$ and choose $H_i$, $G_i$ for $i=1,2$ as previously described in this subsection.
By (\ref{EqPw2}), (\ref{EqLw2}) and Lemma \ref{H12G12g5b}, all the hypothesis of Theorem \ref{AmalgamThm} are satisfied.
Thus, the graphs $B^*_q(S,T,u)$ are $(q+3-u)$--regular of girth $5$
and order $2(q^2-u-1)$ for each $0 \le u \le q-1$. Note that the girth of
$B^*_q(S,T,u)$ is exactly $5$ because $H_1$ has girth $5$ (cf. Remark \ref{sharp5}).
\end{proof}


\section{Small Cases}\label{smallcases}

We now present some constructions of graphs $B^*_q(S,T,u)$ for small
prime values of $q$. The first two constructions complete the proof of
Theorem \ref{p1mod6}.

\subsection{$q=13$}\label{q13section}

In this case, let $S=T=\{0\}$, $H_1, H_2, G_1$ and $G_2$ as in Figure \ref{q13}. 
The graphs $G_i$ are obtained from $H_i$ deleting two independent edges satisfying 
the hypothesis of Lemma \ref{TchInd} and joining all their end--vertices to a new 
vertex, say $0$, for $i=1,2$.
Specifically
$G_1=H_1-\{(1,10),(3,12)\} \cup \{(0,1),(0,3),(0,10),(0,12)\},$
$G_2=H_2-\{(2,8),(5,11)\} \cup \{(0,2),(0,8),$ $(0,5),(0,11)\}$ 
and as unlabeled graphs $G_1$ is isomorphic to $G_2$.
Hence, the graphs $G_1$ and $G_2$ have order $13$, girth $5$ and are
bi--regular with one vertex of degree four and all other vertices of degree three.

\setlength{\intextsep}{10pt}
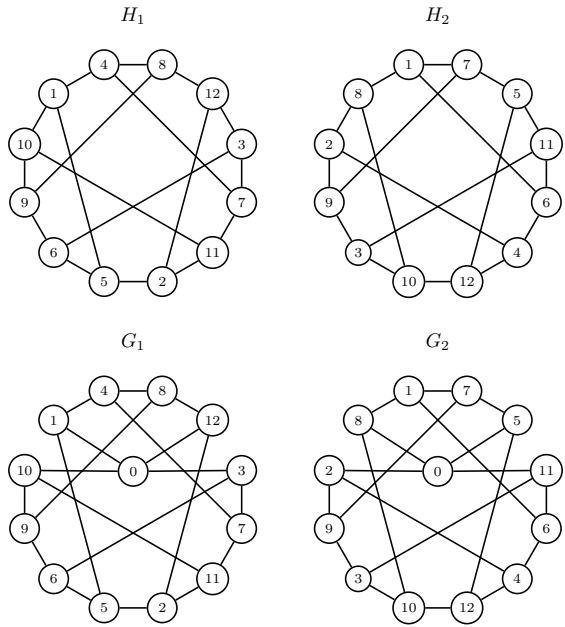
\begin{figure}[ht]
\begin{center}
\resizebox{!}{4.2cm}{
\begin{minipage}{16cm}
\begin{center}
\begin{tabular}{cc}
$H_1$ & $H_2$\\
      &      \\
\begin{pspicture}(-2.5,-2)(2.5,2)
\psset{unit=2}
\SpecialCoor
\cnodeput(1;135){N1}{\scriptsize $\, 1 \,$}
\cnodeput(1;105){N4}{\scriptsize $\, 4 \,$}
\cnodeput(1;75){N8}{\scriptsize $\, 8 \,$}
\cnodeput(1;45){N12}{\scriptsize $12$}
\cnodeput(1;15){N3}{\scriptsize $\, 3 \,$}
\cnodeput(1;345){N7}{\scriptsize $\, 7 \,$}
\cnodeput(1;315){N11}{\scriptsize $11$}
\cnodeput(1;285){N2}{\scriptsize $\, 2 \,$}
\cnodeput(1;255){N5}{\scriptsize $\, 5 \,$}
\cnodeput(1;225){N6}{\scriptsize $\, 6 \,$}
\cnodeput(1;195){N9}{\scriptsize $\, 9 \,$}
\cnodeput(1;165){N10}{\scriptsize $10$}
\ncline{N1}{N4}
\ncline{N4}{N8}
\ncline{N8}{N12}
\ncline{N12}{N3}
\ncline{N3}{N7}
\ncline{N7}{N11}
\ncline{N11}{N2}
\ncline{N2}{N5}
\ncline{N5}{N6}
\ncline{N6}{N9}
\ncline{N9}{N10}
\ncline{N10}{N1}
\ncline{N1}{N5}
\ncline{N2}{N12}
\ncline{N3}{N6}
\ncline{N4}{N7}
\ncline{N8}{N9}
\ncline{N10}{N11}
\end{pspicture} &
\begin{pspicture}(-2.5,-2)(2.5,2)
\psset{unit=2}
\SpecialCoor
\cnodeput(1;75){N7}{\scriptsize $\, 7 \,$}
\cnodeput(1;45){N5}{\scriptsize $\, 5 \,$}
\cnodeput(1;15){N11}{\scriptsize $11$}
\cnodeput(1;345){N6}{\scriptsize $\, 6 \,$}
\cnodeput(1;315){N4}{\scriptsize $\, 4 \,$}
\cnodeput(1;285){N12}{\scriptsize $12$}
\cnodeput(1;255){N10}{\scriptsize $10$}
\cnodeput(1;225){N3}{\scriptsize $3$}
\cnodeput(1;195){N9}{\scriptsize $\, 9 \,$}
\cnodeput(1;165){N2}{\scriptsize $\, 2 \,$}
\cnodeput(1;135){N8}{\scriptsize $\, 8 \,$}
\cnodeput(1;105){N1}{\scriptsize $\, 1 \,$}
\ncline{N1}{N8}
\ncline{N8}{N2}
\ncline{N2}{N9}
\ncline{N9}{N3}
\ncline{N3}{N10}
\ncline{N10}{N12}
\ncline{N12}{N4}
\ncline{N4}{N6}
\ncline{N6}{N11}
\ncline{N11}{N5}
\ncline{N5}{N7}
\ncline{N7}{N1}
\ncline{N1}{N6}
\ncline{N2}{N4}
\ncline{N3}{N11}
\ncline{N5}{N12}
\ncline{N7}{N9}
\ncline{N8}{N10}
\end{pspicture}\\
      &      \\
$G_1$ & $G_2$\\
      &      \\
\begin{pspicture}(-2.5,-2)(2.5,2)
\psset{unit=2}
\SpecialCoor
\cnodeput(0.25;90){N0}{\scriptsize $\, 0 \,$}
\cnodeput(1;135){N1}{\scriptsize $\, 1 \,$}
\cnodeput(1;105){N4}{\scriptsize $\, 4 \,$}
\cnodeput(1;75){N8}{\scriptsize $\, 8 \,$}
\cnodeput(1;45){N12}{\scriptsize $12$}
\cnodeput(1;15){N3}{\scriptsize $\, 3 \,$}
\cnodeput(1;345){N7}{\scriptsize $\, 7 \,$}
\cnodeput(1;315){N11}{\scriptsize $11$}
\cnodeput(1;285){N2}{\scriptsize $\, 2 \,$}
\cnodeput(1;255){N5}{\scriptsize $\, 5 \,$}
\cnodeput(1;225){N6}{\scriptsize $\, 6 \,$}
\cnodeput(1;195){N9}{\scriptsize $\, 9 \,$}
\cnodeput(1;165){N10}{\scriptsize $10$}
\ncline{N1}{N4}
\ncline{N4}{N8}
\ncline{N8}{N12}
\ncline{N3}{N7}
\ncline{N7}{N11}
\ncline{N11}{N2}
\ncline{N2}{N5}
\ncline{N5}{N6}
\ncline{N6}{N9}
\ncline{N9}{N10}
\ncline{N1}{N5}
\ncline{N2}{N12}
\ncline{N3}{N6}
\ncline{N4}{N7}
\ncline{N8}{N9}
\ncline{N10}{N11}
\ncline{N0}{N1}
\ncline{N0}{N3}
\ncline{N0}{N10}
\ncline{N0}{N12}
\end{pspicture} &
\begin{pspicture}(-2.5,-2)(2.5,2)
\psset{unit=2}
\SpecialCoor
\cnodeput(0.25;90){N0}{\scriptsize $\, 0 \,$}
\cnodeput(1;75){N7}{\scriptsize $\, 7 \,$}
\cnodeput(1;45){N5}{\scriptsize $\, 5 \,$}
\cnodeput(1;15){N11}{\scriptsize $11$}
\cnodeput(1;345){N6}{\scriptsize $\, 6 \,$}
\cnodeput(1;315){N4}{\scriptsize $\, 4 \,$}
\cnodeput(1;285){N12}{\scriptsize $12$}
\cnodeput(1;255){N10}{\scriptsize $10$}
\cnodeput(1;225){N3}{\scriptsize $3$}
\cnodeput(1;195){N9}{\scriptsize $\, 9 \,$}
\cnodeput(1;165){N2}{\scriptsize $\, 2 \,$}
\cnodeput(1;135){N8}{\scriptsize $\, 8 \,$}
\cnodeput(1;105){N1}{\scriptsize $\, 1 \,$}
\ncline{N1}{N8}
\ncline{N2}{N9}
\ncline{N9}{N3}
\ncline{N3}{N10}
\ncline{N10}{N12}
\ncline{N12}{N4}
\ncline{N4}{N6}
\ncline{N6}{N11}
\ncline{N5}{N7}
\ncline{N7}{N1}
\ncline{N1}{N6}
\ncline{N2}{N4}
\ncline{N3}{N11}
\ncline{N5}{N12}
\ncline{N7}{N9}
\ncline{N8}{N10}
\ncline{N0}{N5}
\ncline{N0}{N2}
\ncline{N0}{N8}
\ncline{N0}{N11}
\end{pspicture}
\\
\end{tabular}
\end{center}
\end{minipage}
}
\caption{The graphs $H_i$ and $G_i$ for $i\in{1,2}$ and $q=13$.}
\label{q13}
\end{center}
\end{figure}

Note that as unlabeled graphs $H_1$ is isomorphic to $H_2$ and they are both isomorphic to one of the two
cubic graphs on $12$ vertices of girth $5$, specifically {\em $12$ cubic graph $84$} from
\cite{M09, WolfAlpha}.

    \begin{lemma}\label{q13lemma}
Let $S=T=\{0\}$, $H_1, H_2, G_1$ and $G_2$ as described above. Then the graph $B^*_{13}(0,0,u)$ is a $(16-u)$--regular
graph of girth $5$  and order $336-26u$, for $0 \le u \le q-1$,.
    \end{lemma}

    \begin{proof}
The weights of these graphs are ${\cal{P}}_{\omega}=\pm \{ 1, 3, 4\}$ and
${\cal{L}}_{\omega}=\pm \{ 2, 5, 6\}$. Thus, by Theorem \ref{AmalgamThm},
the graph $B^*_{13}(0,0,u)$ is a $(16-u)$--regular
graph of girth $5$  and order $26(13-u)-2=336-26u$, for $0 \le u \le q-1$.
    \end{proof}

 \begin{itemize*}
 \item For $u=0$, we obtain a $16$--regular graph of girth $5$ and order $336$,
 with exactly the same order as the $(16,5)$--graphs that appear in $\cite{J05,F10}$;
 \item  for $u=1$, we obtain a $15$--regular graph of girth $5$ and $310$ vertices
which has two vertices less than the $(15,5)$--graphs that appear in  $\cite{J05,F10}$;
 \item for $u=2$ we obtain a $14$--regular graph of girth $5$ and $284$ vertices which has
 four vertices less than the $(14,5)$--graph in $\cite{J05}$;
 \end{itemize*}

\subsection{$q=19$}

Let $S=T=\{0\}$ and let $H_1, H_2, G_1$ and $G_2$ be as in Figure \ref{q19}.
The graphs $G_i$ are obtained from $H_i$ deleting two independent edges satisfying the
hypothesis of Lemma \ref{TchInd} and joining all their end--vertices to a new vertex, say $0$, for $i=1,2$.
Specifically
$G_1=H_1-\{(1,10),(9,16)\} \cup \{(0,1),(0,9),(0,10),(0,16)\}$ and
$G_2=H_2-\{(8,13),(11,15)\} \cup \{(0,8),(0,13),(0,11),(0,15)\}$.
Hence, the graphs $G_1$ and $G_2$ have order $19$, girth $5$ and are
bi--regular with one vertex of degree four and all other vertices of degree $3$.


\

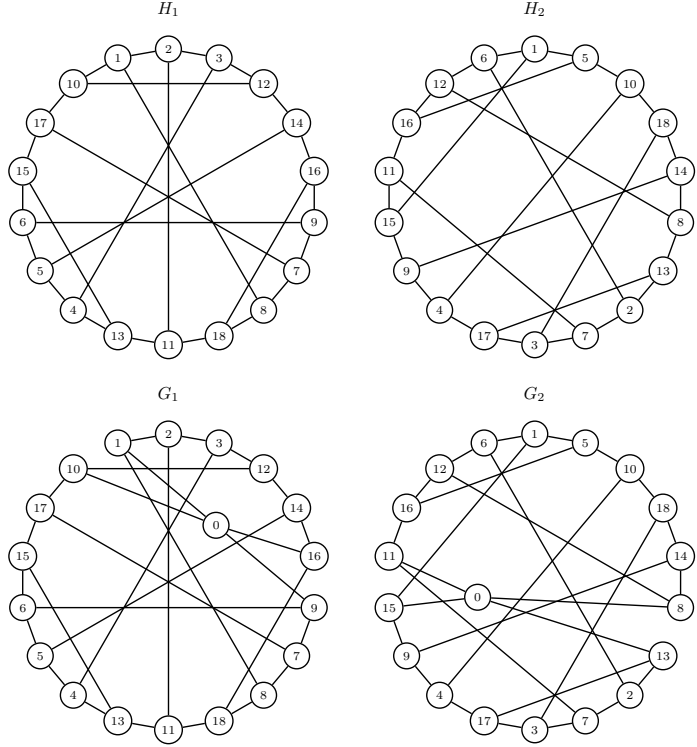
\begin{figure}[h]
\begin{center}
\resizebox{!}{5cm}{
\begin{minipage}{15cm}
\begin{tabular}{cc}
$H_1$ & $H_2$\\
      &      \\
\begin{pspicture}(-3.5,-3)(3.5,3)
\psset{unit=3}
\SpecialCoor
\cnodeput(1;110){N1}{\scriptsize $\, 1 \,$}
\cnodeput(1;90){N2}{\scriptsize $\, 2 \,$}
\cnodeput(1;70){N3}{\scriptsize $\, 3 \,$}
\cnodeput(1;50){N12}{\scriptsize $12$}
\cnodeput(1;30){N14}{\scriptsize $14$}
\cnodeput(1;10){N16}{\scriptsize $16$}
\cnodeput(1;350){N9}{\scriptsize $\, 9 \,$}
\cnodeput(1;330){N7}{\scriptsize $\, 7 \,$}
\cnodeput(1;310){N8}{\scriptsize $\, 8 \,$}
\cnodeput(1;290){N18}{\scriptsize $18$}
\cnodeput(1;270){N11}{\scriptsize $11$}
\cnodeput(1;250){N13}{\scriptsize $13$}
\cnodeput(1;230){N4}{\scriptsize $\, 4 \,$}
\cnodeput(1;210){N5}{\scriptsize $\, 5 \,$}
\cnodeput(1;190){N6}{\scriptsize $\, 6 \,$}
\cnodeput(1;170){N15}{\scriptsize $15$}
\cnodeput(1;150){N17}{\scriptsize $17$}
\cnodeput(1;130){N10}{\scriptsize $10$}
\ncline{N1}{N2}
\ncline{N2}{N3}
\ncline{N3}{N12}
\ncline{N12}{N14}
\ncline{N14}{N16}
\ncline{N16}{N9}
\ncline{N9}{N7}
\ncline{N7}{N8}
\ncline{N8}{N18}
\ncline{N18}{N11}
\ncline{N11}{N13}
\ncline{N13}{N4}
\ncline{N4}{N5}
\ncline{N5}{N6}
\ncline{N6}{N15}
\ncline{N15}{N17}
\ncline{N17}{N10}
\ncline{N10}{N1}
\ncline{N1}{N8}
\ncline{N2}{N11}
\ncline{N3}{N4}
\ncline{N5}{N14}
\ncline{N6}{N9}
\ncline{N7}{N17}
\ncline{N10}{N12}
\ncline{N13}{N15}
\ncline{N16}{N18}
\end{pspicture} & \begin{pspicture}(-3.5,-3)(3.5,3)
\psset{unit=3}
\SpecialCoor
\cnodeput(1;90){N1}{\scriptsize $\, 1 \,$}
\cnodeput(1;70){N5}{\scriptsize $\, 5 \,$}
\cnodeput(1;50){N10}{\scriptsize $10$}
\cnodeput(1;30){N18}{\scriptsize $18$}
\cnodeput(1;10){N14}{\scriptsize $14$}
\cnodeput(1;350){N8}{\scriptsize $\, 8 \,$}
\cnodeput(1;330){N13}{\scriptsize $13$}
\cnodeput(1;310){N2}{\scriptsize $\, 2 \,$}
\cnodeput(1;290){N7}{\scriptsize $\, 7 \,$}
\cnodeput(1;270){N3}{\scriptsize $\, 3 \,$}
\cnodeput(1;250){N17}{\scriptsize $17$}
\cnodeput(1;230){N4}{\scriptsize $\, 4 \,$}
\cnodeput(1;210){N9}{\scriptsize $\, 9 \,$}
\cnodeput(1;190){N15}{\scriptsize $15$}
\cnodeput(1;170){N11}{\scriptsize $11$}
\cnodeput(1;150){N16}{\scriptsize $16$}
\cnodeput(1;130){N12}{\scriptsize $12$}
\cnodeput(1;110){N6}{\scriptsize $\, 6 \,$}
\ncline{N1}{N5}
\ncline{N5}{N10}
\ncline{N10}{N4}
\ncline{N4}{N17}
\ncline{N17}{N13}
\ncline{N13}{N8}
\ncline{N8}{N12}
\ncline{N12}{N16}
\ncline{N16}{N11}
\ncline{N11}{N15}
\ncline{N15}{N9}
\ncline{N9}{N14}
\ncline{N14}{N18}
\ncline{N18}{N3}
\ncline{N3}{N7}
\ncline{N7}{N2}
\ncline{N2}{N6}
\ncline{N6}{N1}
\ncline{N1}{N15}
\ncline{N2}{N13}
\ncline{N3}{N17}
\ncline{N4}{N9}
\ncline{N5}{N16}
\ncline{N6}{N12}
\ncline{N7}{N11}
\ncline{N8}{N14}
\ncline{N10}{N18}
\end{pspicture}\\
      &      \\
$G_1$ & $G_2$\\
      &      \\
\begin{pspicture}(-3.5,-3)(3.5,3)
\psset{unit=3}
\SpecialCoor
\cnodeput(0.5;50){N0}{\scriptsize $\, 0 \,$}
\cnodeput(1;110){N1}{\scriptsize $\, 1 \,$}
\cnodeput(1;90){N2}{\scriptsize $\, 2 \,$}
\cnodeput(1;70){N3}{\scriptsize $\, 3 \,$}
\cnodeput(1;50){N12}{\scriptsize $12$}
\cnodeput(1;30){N14}{\scriptsize $14$}
\cnodeput(1;10){N16}{\scriptsize $16$}
\cnodeput(1;350){N9}{\scriptsize $\, 9 \,$}
\cnodeput(1;330){N7}{\scriptsize $\, 7 \,$}
\cnodeput(1;310){N8}{\scriptsize $\, 8 \,$}
\cnodeput(1;290){N18}{\scriptsize $18$}
\cnodeput(1;270){N11}{\scriptsize $11$}
\cnodeput(1;250){N13}{\scriptsize $13$}
\cnodeput(1;230){N4}{\scriptsize $\, 4 \,$}
\cnodeput(1;210){N5}{\scriptsize $\, 5 \,$}
\cnodeput(1;190){N6}{\scriptsize $\, 6 \,$}
\cnodeput(1;170){N15}{\scriptsize $15$}
\cnodeput(1;150){N17}{\scriptsize $17$}
\cnodeput(1;130){N10}{\scriptsize $10$}
\ncline{N1}{N2}
\ncline{N2}{N3}
\ncline{N3}{N12}
\ncline{N12}{N14}
\ncline{N14}{N16}
\ncline{N9}{N7}
\ncline{N7}{N8}
\ncline{N8}{N18}
\ncline{N18}{N11}
\ncline{N11}{N13}
\ncline{N13}{N4}
\ncline{N4}{N5}
\ncline{N5}{N6}
\ncline{N6}{N15}
\ncline{N15}{N17}
\ncline{N17}{N10}
\ncline{N1}{N8}
\ncline{N2}{N11}
\ncline{N3}{N4}
\ncline{N5}{N14}
\ncline{N6}{N9}
\ncline{N7}{N17}
\ncline{N10}{N12}
\ncline{N13}{N15}
\ncline{N16}{N18}
\ncline{N0}{N1}
\ncline{N0}{N9}
\ncline{N0}{N10}
\ncline{N0}{N16}
\end{pspicture} & \begin{pspicture}(-3.5,-3)(3.5,3)
\psset{unit=3}
\SpecialCoor
\cnodeput(0.4;195){N0}{\scriptsize $\, 0 \,$}
\cnodeput(1;90){N1}{\scriptsize $\, 1 \,$}
\cnodeput(1;70){N5}{\scriptsize $\, 5 \,$}
\cnodeput(1;50){N10}{\scriptsize $10$}
\cnodeput(1;30){N18}{\scriptsize $18$}
\cnodeput(1;10){N14}{\scriptsize $14$}
\cnodeput(1;350){N8}{\scriptsize $\, 8 \,$}
\cnodeput(1;330){N13}{\scriptsize $13$}
\cnodeput(1;310){N2}{\scriptsize $\, 2 \,$}
\cnodeput(1;290){N7}{\scriptsize $\, 7 \,$}
\cnodeput(1;270){N3}{\scriptsize $\, 3 \,$}
\cnodeput(1;250){N17}{\scriptsize $17$}
\cnodeput(1;230){N4}{\scriptsize $\, 4 \,$}
\cnodeput(1;210){N9}{\scriptsize $\, 9 \,$}
\cnodeput(1;190){N15}{\scriptsize $15$}
\cnodeput(1;170){N11}{\scriptsize $11$}
\cnodeput(1;150){N16}{\scriptsize $16$}
\cnodeput(1;130){N12}{\scriptsize $12$}
\cnodeput(1;110){N6}{\scriptsize $\, 6 \,$}
\ncline{N1}{N5}
\ncline{N5}{N10}
\ncline{N10}{N4}
\ncline{N4}{N17}
\ncline{N17}{N13}
\ncline{N8}{N12}
\ncline{N12}{N16}
\ncline{N16}{N11}
\ncline{N15}{N9}
\ncline{N9}{N14}
\ncline{N14}{N18}
\ncline{N18}{N3}
\ncline{N3}{N7}
\ncline{N7}{N2}
\ncline{N2}{N6}
\ncline{N6}{N1}
\ncline{N1}{N15}
\ncline{N2}{N13}
\ncline{N3}{N17}
\ncline{N4}{N9}
\ncline{N5}{N16}
\ncline{N6}{N12}
\ncline{N7}{N11}
\ncline{N8}{N14}
\ncline{N10}{N18}
\ncline{N0}{N8}
\ncline{N0}{N11}
\ncline{N0}{N13}
\ncline{N0}{N15}
\end{pspicture}
\\
\end{tabular}
\end{minipage}
}
\caption{The graphs $H_i$ and $G_i$ for $i\in{1,2}$ and $q=19$.}
\label{q19}
\end{center}
\end{figure}


    \begin{lemma}\label{q19lemma}
Let $S=T=\{0\}$, $H_1, H_2, G_1$ and $G_2$ be as described above. Then the graph $B^*_{19}(0,0,u)$ is a $(22-u)$--regular
graph of girth $5$  and order $720-38u$, for $0 \le u \le q-1$,.
    \end{lemma}

    \begin{proof}
The weights of these graphs are ${\cal{P}}_w=\pm \{ 1, 2, 3,  7,  9\}$ and
   ${\cal{L}}_w=\pm \{ 4, 5, 6,  8\}$. Thus, by Theorem \ref{AmalgamThm},
the graph $B^*_{19}(0,0,u)$ is a $(22-u)$--regular
graph of girth $5$ and order $38(19-u)-2=720-38u$, for $0 \le u \le q-1$.
    \end{proof}

     \begin{itemize*}
 \item For $u=0$, we obtain a $22$--regular graph of girth $5$ and order $720$,
 with exactly the same order as the $(22,5)$--graphs that appear in $\cite{J05,F10}$;
 \item  for $u=1$, we obtain a $21$--regular graph of girth $5$ and $682$ vertices
which has two vertices less than the $(21,5)$--graphs that appear in  $\cite{J05,F10}$;
 \end{itemize*}

\subsection{$q=11$}\label{q11section}

For $q=11$ we are going to remove $6$ vertices from $B_{11}$ instead of $2$, but we will construct a
$(q+2)$--regular graph instead of a $(q+3)$--regular one.


    \begin{lemma}\label{q11-lemma}
Let $S=\{0,1,2,4,6,8\}$ and $T=\emptyset$. Let $H_1=(3,5,10,7,9)$ be a $5$--circuit with weights $\pm \{2, 3, 5\}$,
$G_1=(0,2,4,6,8,10,1,3,5,7,9)$ be a $11$--circuit with weight $\{\pm 2\}$,
and $H_2=G_2 =(0,1,2,3,4,5,6,7,8,9,10) \cup (0,4) \cup (2,6) \cup (1,8)$ be a $11$--circuit with three chords
and weights $\pm \{1, 4\}$ (see Figure \ref{q11}). Then the graph $B^*_{11}(S,T,u)$ is a
$(13-u)$--regular graph of girth $5$ and order $22(11-u)-6=236-22u$, for $u \le q-1$.
In particular, we obtain a $13$-regular graph of girth $5$ and order $236$ for $u=0$.
    \end{lemma}

\setlength{\intextsep}{0cm}
\setlength{\abovecaptionskip}{-0.5cm}
\setlength{\belowcaptionskip}{-0.5cm}
\begin{figure}[h]
\begin{center}
\resizebox{!}{2cm}{
\begin{minipage}{5cm}
\begin{pspicture}(-3,-3)(3,3)
\psset{unit=2}
\SpecialCoor
\cnodeput(1;90){N0}{\scriptsize $0$}
\cnodeput(1;57.28){N1}{\scriptsize $1$}
\cnodeput(1;24.56){N2}{\scriptsize $2$}
\cnodeput(1;351.84){N3}{\scriptsize $3$}
\cnodeput(1;319.12){N4}{\scriptsize $4$}
\cnodeput(1;286.4){N5}{\scriptsize $5$}
\cnodeput(1;253.68){N6}{\scriptsize $6$}
\cnodeput(1;220.96){N7}{\scriptsize $7$}
\cnodeput(1;188.24){N8}{\scriptsize $8$}
\cnodeput(1;155.52){N9}{\scriptsize $9$}
\cnodeput(1;122.8){N10}{\scriptsize $10$}
\ncline{N0}{N1}
\ncline{N1}{N2}
\ncline{N2}{N3}
\ncline{N3}{N4}
\ncline{N4}{N5}
\ncline{N5}{N6}
\ncline{N6}{N7}
\ncline{N7}{N8}
\ncline{N8}{N9}
\ncline{N9}{N10}
\ncline{N10}{N0}
\ncline{N0}{N4}
\ncline{N1}{N8}
\ncline{N2}{N6}
\end{pspicture}
\end{minipage}
}
\caption{The graphs $H_2=G_2$ for $q=11$.}
\label{q11}
\end{center}
\end{figure}
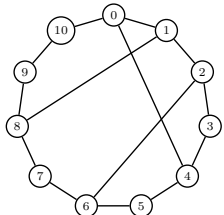

\

    \begin{proof}
Since ${\cal{P}}_{\omega}=\pm \{2, 3, 5\}$ and ${\cal{L}}_{\omega}=\pm \{1, 4\}$,
the thesis follows by Theorem \ref{AmalgamThm}.
    \end{proof}

Note that the graph $B^*_{11}(S,T,0)$ has four vertices less than those constructed in $\cite{J05,F10}$.



\subsection{$q=17$}\label{q17section}

For $q=17$ we are going to remove $6$ vertices instead of $2$
and construct a $(q+3)$--regular graph, obtaining a better result than
the one obtained in \cite{B98}.

    \begin{lemma}\label{q17lemma}
Let $S=T=\{7,10,12\}$, $H_1, H_2, G_1$ and $G_2$ as in Figure \ref{q17}.
The graphs $G_1$ and $G_2$ have order $17$, girth $5$ and are bi--regular
with three vertices of degree four and all other vertices of degree $3$.
Then the graph $B^*_{17}(S,T,u)$ is a $(20-u)$--regular
graph of girth $5$  and order $572-34u$, for $u \ge q-1$.
    \end{lemma}

    \begin{proof}
In this case  ${\cal{P}}_w=\pm \{1, 3, 4, 5\}$ and ${\cal{L}}_w=\pm \{ 2, 6, 7, 8\}$,
thus, by Theorem \ref{AmalgamThm}, the graph $B^*_{17}(S,T,u)$ is a $(20-u)$--regular
graph of girth $5$ and order $34(17-u)-6=572-34u$ for $u \ge q-1$.
    \end{proof}

In $\cite{J05}$ the author constructs $(k,5)$--graphs of order $32(k-2)$, while we have constructed
$(k,5)$--graphs of order $34(k-3)$ which have $44-2k$ fewer vertices, for $k \in \{4, \ldots, 20\}$.
In particular, we obtain a $20$-regular graph of girth $5$ and order $572$ for $u=0$ which has four vertices less than the one constructed in $\cite{J05}$. Note also that as unlabeled graphs $H_1 \cong H_2$ and they are both isomorphic to the Heawood graph.

\setlength{\intextsep}{10pt}
\setlength{\abovecaptionskip}{1cm}
\setlength{\belowcaptionskip}{-1cm}
\begin{figure}[h]
\begin{center}
\resizebox{!}{4cm}{
\begin{minipage}{15cm}
\begin{tabular}{cc}
$H_1$ & $H_2$\\
      &      \\
\begin{pspicture}(-3.5,-3)(3.5,3)
\psset{unit=3}
\SpecialCoor
\cnodeput(1;90){N0}{\scriptsize $\, 0 \,$}
\cnodeput(1;64.29){N16}{\scriptsize $16$}
\cnodeput(1;38.58){N15}{\scriptsize $15$}
\cnodeput(1;12.87){N2}{\scriptsize $\, 2 \,$}
\cnodeput(1;347.16){N1}{\scriptsize $\, 1 \,$}
\cnodeput(1;321.45){N13}{\scriptsize $13$}
\cnodeput(1;295.74){N14}{\scriptsize $14$}
\cnodeput(1;270.03){N11}{\scriptsize $11$}
\cnodeput(1;244.32){N8}{\scriptsize $\, 8 \,$}
\cnodeput(1;218.61){N5}{\scriptsize $\, 5 \,$}
\cnodeput(1;192.9){N4}{\scriptsize $\, 4 \,$}
\cnodeput(1;167.19){N9}{\scriptsize $\, 9 \,$}
\cnodeput(1;141.48){N6}{\scriptsize $\, 6 \,$}
\cnodeput(1;115.77){N3}{\scriptsize $\, 3 \,$}
\ncline{N0}{N16}
\ncline{N16}{N15}
\ncline{N15}{N2}
\ncline{N2}{N1}
\ncline{N1}{N13}
\ncline{N13}{N14}
\ncline{N14}{N11}
\ncline{N11}{N8}
\ncline{N8}{N5}
\ncline{N5}{N4}
\ncline{N4}{N9}
\ncline{N9}{N6}
\ncline{N6}{N3}
\ncline{N3}{N0}
\ncline{N0}{N13}
\ncline{N15}{N11}
\ncline{N1}{N5}
\ncline{N14}{N9}
\ncline{N8}{N3}
\ncline{N4}{N16}
\ncline{N6}{N2}
\end{pspicture} & \begin{pspicture}(-3.5,-3)(3.5,3)
\psset{unit=3}
\SpecialCoor
\cnodeput(1;90){N0}{\scriptsize $\, 0 \,$}
\cnodeput(1;64.29){N2}{\scriptsize $\, 2 \,$}
\cnodeput(1;38.58){N4}{\scriptsize $\, 4 \,$}
\cnodeput(1;12.87){N14}{\scriptsize $14$}
\cnodeput(1;347.16){N3}{\scriptsize $\, 3 \,$}
\cnodeput(1;321.45){N9}{\scriptsize $\, 9 \,$}
\cnodeput(1;295.74){N16}{\scriptsize $16$}
\cnodeput(1;270.03){N6}{\scriptsize $\, 6 \,$}
\cnodeput(1;244.32){N13}{\scriptsize $13$}
\cnodeput(1;218.61){N5}{\scriptsize $\, 5 \,$}
\cnodeput(1;192.9){N11}{\scriptsize $11$}
\cnodeput(1;167.19){N1}{\scriptsize $\, 1 \,$}
\cnodeput(1;141.48){N8}{\scriptsize $\, 8 \,$}
\cnodeput(1;115.77){N15}{\scriptsize $15$}
\ncline{N0}{N2}
\ncline{N2}{N4}
\ncline{N4}{N14}
\ncline{N14}{N3}
\ncline{N3}{N9}
\ncline{N9}{N16}
\ncline{N16}{N6}
\ncline{N6}{N13}
\ncline{N13}{N5}
\ncline{N5}{N11}
\ncline{N11}{N1}
\ncline{N1}{N8}
\ncline{N8}{N15}
\ncline{N15}{N0}
\ncline{N0}{N9}
\ncline{N4}{N6}
\ncline{N3}{N5}
\ncline{N16}{N1}
\ncline{N13}{N15}
\ncline{N11}{N2}
\ncline{N8}{N14}
\end{pspicture}\\
      &      \\
$G_1$ & $G_2$\\
      &      \\
\begin{pspicture}(-3.5,-3)(3.5,3)
\psset{unit=3}
\SpecialCoor
\cnodeput(1;90){N0}{\scriptsize $\, 0 \,$}
\cnodeput(1;68.824){N1}{\scriptsize $\, 1 \,$}
\cnodeput(1;47.648){N15}{\scriptsize $15$}
\cnodeput(1;26.472){N16}{\scriptsize $16$}
\cnodeput(1;5.296){N13}{\scriptsize $13$}
\cnodeput(1;344.12){N12}{\scriptsize $12$}
\cnodeput(1;322.944){N11}{\scriptsize $11$}
\cnodeput(1;301.768){N10}{\scriptsize $10$}
\cnodeput(1;280.592){N9}{\scriptsize $\, 9 \,$}
\cnodeput(1;259.416){N4}{\scriptsize $\, 4 \,$}
\cnodeput(1;238.24){N3}{\scriptsize $\, 3 \,$}
\cnodeput(1;217.064){N8}{\scriptsize $\, 8 \,$}
\cnodeput(1;195.888){N5}{\scriptsize $\, 5 \,$}
\cnodeput(1;174.172){N6}{\scriptsize $\, 6 \,$}
\cnodeput(1;153.536){N7}{\scriptsize $\, 7 \,$}
\cnodeput(1;132.36){N2}{\scriptsize $\, 2 \,$}
\cnodeput(1;111.184){N14}{\scriptsize $14$}
\ncline{N0}{N1}
\ncline{N1}{N15}
\ncline{N15}{N16}
\ncline{N16}{N13}
\ncline{N13}{N12}
\ncline{N12}{N11}
\ncline{N11}{N10}
\ncline{N10}{N9}
\ncline{N9}{N4}
\ncline{N4}{N3}
\ncline{N3}{N8}
\ncline{N8}{N5}
\ncline{N5}{N6}
\ncline{N6}{N7}
\ncline{N7}{N2}
\ncline{N2}{N14}
\ncline{N14}{N0}
\ncline{N0}{N12}
\ncline{N1}{N6}
\ncline{N2}{N16}
\ncline{N3}{N15}
\ncline{N4}{N7}
\ncline{N5}{N10}
\ncline{N7}{N11}
\ncline{N8}{N12}
\ncline{N9}{N13}
\ncline{N10}{N14}
\end{pspicture} & \begin{pspicture}(-3.5,-3)(3.5,3)
\psset{unit=3}
\SpecialCoor
\cnodeput(1;90){N0}{\scriptsize $\, 0 \,$}
\cnodeput(1;68.824){N2}{\scriptsize $\, 2 \,$}
\cnodeput(1;47.648){N12}{\scriptsize $12$}
\cnodeput(1;26.472){N5}{\scriptsize $\, 5 \,$}
\cnodeput(1;5.296){N3}{\scriptsize $\, 3 \,$}
\cnodeput(1;344.12){N14}{\scriptsize $14$}
\cnodeput(1;322.944){N4}{\scriptsize $\, 4 \,$}
\cnodeput(1;301.768){N6}{\scriptsize $\, 6 \,$}
\cnodeput(1;280.592){N8}{\scriptsize $\, 8 \,$}
\cnodeput(1;259.416){N10}{\scriptsize $10$}
\cnodeput(1;238.24){N1}{\scriptsize $\, 1 \,$}
\cnodeput(1;217.064){N7}{\scriptsize $\, 7 \,$}
\cnodeput(1;195.888){N16}{\scriptsize $16$}
\cnodeput(1;174.172){N9}{\scriptsize $\, 9 \,$}
\cnodeput(1;153.536){N11}{\scriptsize $11$}
\cnodeput(1;132.36){N13}{\scriptsize $13$}
\cnodeput(1;111.184){N15}{\scriptsize $15$}
\ncline{N0}{N2}
\ncline{N1}{N12}
\ncline{N2}{N12}
\ncline{N12}{N5}
\ncline{N5}{N3}
\ncline{N3}{N14}
\ncline{N14}{N4}
\ncline{N4}{N6}
\ncline{N6}{N8}
\ncline{N8}{N10}
\ncline{N10}{N1}
\ncline{N1}{N7}
\ncline{N7}{N16}
\ncline{N16}{N9}
\ncline{N9}{N11}
\ncline{N11}{N13}
\ncline{N13}{N15}
\ncline{N15}{N0}
\ncline{N0}{N10}
\ncline{N2}{N9}
\ncline{N3}{N10}
\ncline{N4}{N11}
\ncline{N5}{N13}
\ncline{N6}{N12}
\ncline{N7}{N14}
\ncline{N7}{N15}
\ncline{N8}{N16}
\end{pspicture}
\\
\end{tabular}
\end{minipage}
}
\caption{The graphs $H_i$ and $G_i$ for $i\in{1,2}$ and $q=17$.}
\label{q17}
\end{center}
\end{figure}
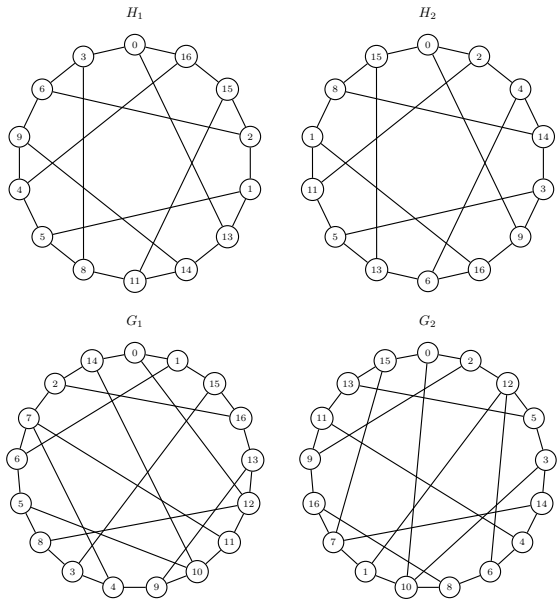

\subsection*{Acknowledgment}
\footnotesize
{ Research   supported by the Ministerio de Educaci\'on y Ciencia,
Spain, the European Regional Development Fund (ERDF) under
project MTM2008-06620-C03-02, CONACyT-M\'exico under project 57371 and PAPIIT-M\'exico under project 104609-3.}\\


\end{document}